\documentclass[a4paper,11pt,twoside]{article}
\usepackage[utf8x]{inputenc}
\usepackage[english]{babel}

\usepackage[margin=2cm]{geometry}

\usepackage{isomath}

\usepackage{enumitem}
    
\usepackage{float}

\usepackage{graphicx}

\usepackage{tabulary}

\usepackage{tensor}

\usepackage{url,epsfig}

\usepackage{amsmath,amsfonts,amssymb,amsthm}

\usepackage{bm}

\usepackage{mathtools}
\usepackage{bbm}
\usepackage{latexsym}

\usepackage{enumitem}

\usepackage[ocgcolorlinks]{hyperref} 

\usepackage{mathrsfs}   

\theoremstyle{plain}
\newtheorem{teorema}{Theorem}[section]

\newtheorem{lemma}[teorema]{Lemma}
\newtheorem{corollario}[teorema]{Corollary}
\newtheorem*{teorema*}{Theorem}
\newtheorem*{proposizione*}{Proposition}
\newtheorem*{lemma*}{Lemma}
\newtheorem*{corollario*}{Corollary}

\theoremstyle{definition}

\newtheorem*{definizione*}{Definition}
\newtheorem*{esempio*}{Example}

\theoremstyle{remark}
\newtheorem{nota}[teorema]{Remark}

\newtheorem*{nota*}{Remark}
\newtheorem*{osservazione*}{Remark}
\newtheorem*{esercizio*}{Esercizio}

\newcommand{\m}[1]{\mathcal{#1}}
\newcommand{\bb}[1]{\mathbb{#1}}

\newcommand{\mrm}[1]{\mathrm{#1}}

\newcommand{\scr}[1]{\mathscr{#1}}

\newcommand{\diff}{\partial}         
\newcommand{\bdiff}{\bar{\partial}} 

\newcommand{\cotJ}{T^*\!\!\!\scr{J}}

\DeclarePairedDelimiter\card{\lvert}{\rvert} 
\DeclarePairedDelimiter\norm{\lVert}{\rVert}
\DeclarePairedDelimiter{\set}{\{}{\}}
\newcommand{\tc}{\mathrel{}\mathclose{}\middle|\mathopen{}\mathrel{}}

\numberwithin{equation}{section} 

\title{Solutions to Donaldson's hyperk\"ahler reduction on a curve}
\date{}
\author{Carlo Scarpa and Jacopo Stoppa}

\begin{document}

\maketitle

\abstract{\noindent We study an infinite-dimensional hyperk\"ahler reduction introduced by Donaldson and associated with the constant scalar curvature equation on a Riemann surface. It is known that the corresponding moment map equations admit special solutions constructed from holomorphic quadratic differentials. Here we obtain a more general existence result and so a larger hyperk\"ahler moduli space.}


\section{Introduction}

Let $\Sigma$ be a compact oriented surface of genus $g(\Sigma)>1$, and let $\omega$ be a fixed area form on $\Sigma$. The group $\mathcal{G}$ of exact area-preserving diffeomorphisms acts on the infinite-dimensional manifold $\scr{J}$ of complex structures on $\Sigma$ by pullback, and this action is Hamiltonian, with a moment map $\mu$ given by the Gauss curvature, minus its average. This fact is a special case of well-known results of Fujiki and Donaldson on scalar curvature as a moment map, and was first pointed out by Quillen (see \cite{Donaldson_hyperkahler} Section 2.2). The action of $\mathcal{G}$ on $\scr{J}$ preserves a natural formal K\"ahler structure, and one has a well-defined K\"ahler reduction $\mathcal{M} = \mu^{-1}(0)/\mathcal{G}$. This space $\mathcal{M}$ can be identified with the moduli space of marked Riemann surfaces $(\Sigma, J)$, together with a choice of holomorphic line bundle $L$ of fixed degree. The Teichm\"uller space $\mathcal{T}$ of $\Sigma$ is the quotient of $\mathcal{M}$ by the torus $H^1(\Sigma, \mathbb{R})/H^1(\Sigma, \mathbb{Z})$ (see loc. cit.).
 
Motivated by a clear analogy with the case of Higgs bundles and harmonic metrics, Donaldson \cite{Donaldson_hyperkahler} considered the problem of extending the K\"ahler reduction of $\scr{J}$ described above to a hyperk\"ahler reduction of (an open subset of) the cotangent space $\cotJ$. Indeed $\cotJ$ comes with a natural hyperk\"ahler structure, such that the induced action of $\mathcal{G}$ on it is Hamiltonian with respect to the symplectic forms in the hyperk\"ahler family. The zero-locus equations for the moment maps of this action give a system of equations, generalizing the usual constant Gauss curvature equation on $\Sigma$. 

These equations are given, a bit implicitly, in \cite[Proposition $17$]{Donaldson_hyperkahler}. Taking the equivalent point of view of fixing $J$ and varying $\omega$ in its class they can be spelled out in terms of a smooth quadratic differential $q$ and a K\"ahler form $\omega'\in[\omega]$ on a fixed marked surface $(\Sigma, J)$, with corresponding metric $g'$, yielding the system
\begin{equation}\label{eq:sistema_HCSCK_originale}
\begin{dcases}
\norm{q}^2_{\omega'}<1;\\
{\nabla^{1,0}_{\omega'}}^*{\nabla^{1,0}_{\omega'}}^*q=0;\\
2\,s(\omega')-2\,\widehat{s(\omega')}+\Delta_{\omega'}\,\mrm{log}\left(1+\sqrt{1-\norm{q}_{\omega'}^2}\right) +\mrm{div} \frac{2\mathrm{Re}\,(g'(\bdiff q, \bar{q}))^{\sharp'}}{1+\sqrt{1-\norm{q}_{\omega'}^2}} =0 
\end{dcases}
\end{equation}
(see \cite[\S $4.2$]{ScarpaStoppa_hyperk_reduction}). Here ${\nabla^{1,0}_{\omega'}}^*$ denotes the formal adjoint to the $(1, 0)$--part of the Levi-Civita connection. The two partial differential equations appearing in \eqref{eq:sistema_HCSCK_originale} correspond respectively to the complex and real moment maps for the action of $\mathcal{G}$. 

What we discussed so far is a special case of a hyperk\"ahler reduction that can be formulated for K\"ahler manifolds of any dimension. In our previous work \cite{ScarpaStoppa_hyperk_reduction} we studied these higher dimensional equations, for which we proposed the name \emph{HcscK equations}, focusing on complex surfaces (our work relies on the general results of Biquard and Gauduchon \cite{Biquard_Gauduchon} concerning hyperk\"ahler metrics on cotangent bundles). In the original case of curves, Donaldson was interested in solutions of \eqref{eq:sistema_HCSCK_originale} given by a \emph{holomorphic} quadratic differential $q$, since these special solutions can be used to define a hyperk\"ahler extension of the Weil-Petersson metric on the Teichm\"uller space $\m{T}$ of $\Sigma$ to an open subset of $T^*\m{T}$. Under the assumption that $q$ is holomorphic, the real and complex moment map equations decouple, and \eqref{eq:sistema_HCSCK_originale} becomes
\begin{equation}\label{eq:sistema_HCSCK_Donaldson}
\begin{dcases}
\norm{q}^2_{\omega'}<1;\\
2\,s(\omega')-2\,\widehat{s(\omega')}+\Delta\,\mrm{log}\left(1+\sqrt{1-\norm{q}_{\omega'}^2}\right) =0.
\end{dcases}
\end{equation}
This equation was studied by T. Hodge \cite{Hodge_phd_thesis}, who obtained existence and uniqueness results under an explicit condition on $q$. More recent works on this topic include \cite{traut}. However, from the higher dimensional point of view of \cite{ScarpaStoppa_hyperk_reduction}, the restriction to holomorphic quadratic differentials is not very natural, and this motivates us to consider more general solutions to the original coupled system \eqref{eq:sistema_HCSCK_originale}.    

In order to state our results we fix a marked Riemann surface $(\Sigma, J)$ with genus $g(\Sigma) > 1$ and let $\omega_0 \in [\omega]$ denote the unique K\"ahler form of constant scalar curvature.
\begin{teorema}\label{MainThm} The system \eqref{eq:sistema_HCSCK_originale} admits a set of solutions whose points are in bijection with pairs $(\tau,\beta)$, consisting of a holomorphic quadratic differential $\tau$ and a holomorphic $1$-form $\beta$, such that $\norm{\tau}_{\m{C}^{0,\frac{1}{2}}(\omega_0)} < c_1$, $\norm{\beta}_{\m{C}^{1,\frac{1}{2}}(\omega_0)}<c_2$ for certain $c_1, c_2 > 0$. The constants $c_1, c_2$ depend on $(\Sigma, J)$ only through a few Sobolev and elliptic constants with respect to the hyperbolic metric $\omega_0$.
\end{teorema}
In fact an application of the implicit function theorem would give quite easily the result above form some $c_1, c_2 > 0$, but much of the work here goes into proving the stronger characterization in terms of Sobolev and elliptic constants of the hyperbolic metric. The precise constants which play a role will be made clear in the course of the proof. With a little effort the dependence upon these constants could be made completely explicit. As a consequence this gives a construction of a hyperk\"ahler thickening to an open neighbourhood of the zero section in $T^*\mathcal{M}$ of the K\"ahler metric on the moduli space $\mathcal{M}$, which contains that considered by Donaldson as the locus $\beta = 0$. Moreover this open neighbourhood can be controlled in terms of hyperbolic geometry. Note that $T^*\mathcal{M}$ can be identified with the moduli space of collections consisting of a marked Riemann surface $(\Sigma, J)$ together with a holomorphic line bundle of fixed degree, a holomorphic quadratic differential $\tau$ and a holomorphic $1$-form $\beta$. 
\begin{corollario} There is an open subset of the space of collections $T^*\mathcal{M} = \{[(\Sigma, J, L, \tau, \beta)]\}$, given by the conditions $\norm{\tau}_{\m{C}^{0,\frac{1}{2}}(\omega_0)} < c_1(\omega_0)$, $\norm{\beta}_{\m{C}^{1,\frac{1}{2}}(\omega_0)}<c_2(\omega_0)$, which carries an incomplete hyperk\"ahler structure, induced by the hyperk\"ahler reduction of $\cotJ$ by $\mathcal{G}$.  
\end{corollario}  
The rest of the paper is devoted to a proof of Theorem \ref{MainThm}. We provide here an outline. 

In Section \ref{PrelimSec} we first show that solutions of \eqref{eq:sistema_HCSCK_originale}, if they exist, are parametrised a priori by pairs $(\tau, \beta)$ as above. The pair $(0,0)$ corresponds to the unique hyperbolic metric $\omega_0$. Then, following an idea of Donaldson, we perform a conformal transformation of the unknown metric $\omega'$ which brings the real moment map equation to a much simpler form. But in our case this has the cost of turning the linear complex moment map equation into a more complicated quasi-linear equation. 

In Section \ref{ContinuitySec} we introduce a continuity method for solving these equivalent equations. It is given simply by deforming a given pair $(\tau, \beta)$ to $(t \tau, t\beta)$ for $t\in[0,1]$. In sections \ref{EstimatesSec}, \ref{ClosednessSec} we proceed to establish $\m{C}^{2,\frac{1}{2}}(\omega_0)$ a priori estimates on solutions $\omega_t$, $q_t$, and to show that the condition $\norm{q_t}^2_{\omega_t} < 1$ is closed along the continuity path. The latter fact requires to control the growth of the norm $||\omega_t||_{\m{C}^{0,\frac{1}{2}}(\omega_0)}$, which we can achieve provided the norms $\norm{\tau}_{\m{C}^{0,\frac{1}{2}}(\omega_0)}$, $\norm{\beta}_{\m{C}^{1}(\omega_0)}$ are sufficiently small, depending only on a few Sobolev constants of $\omega_0$, and elliptic constants for the Bochner laplacian $\nabla^*_{\omega_0}\nabla_{\omega_0}$ acting on $1$-forms and the Riemannian laplacian $\Delta_{\omega_0}$ acting on functions. Finally in \ref{OpennessSec} we show that the linearization of the operator corresponding to our equations is an isomorphism. For this we need to take $\norm{\beta}_{\m{C}^{1, \frac{1}{2}}(\omega_0)}$ sufficiently small, again in terms of an elliptic constant for the Riemannian laplacian $\Delta_{\omega_0}$ on functions. Thus our continuity path is also open, and moreover the parametrization by $(\tau, \beta)$ is bijective.

\noindent\textbf{Acknowledgements.} We are grateful to Olivier Biquard for a discussion related to the present paper. 
\section{The HcscK system on a curve}\label{PrelimSec}
We are concerned with the coupled system on a Riemann surface $\Sigma$ of genus $g(\Sigma)>1$
\begin{equation}\label{eq:HcscK_system_curve}
\begin{dcases}
{\nabla^{1,0}}^*{\nabla^{1,0}}^*q=0\\
2\,s(\omega)-2\,\widehat{s}+\Delta\left(\mrm{log}\left(1+\sqrt{1-\norm{q}^2}\right)\right)+\mrm{div}\left(\frac{g(\nabla^1 q,\bar{q})}{1+\sqrt{1-\norm{q}^2}}\diff_z+\mrm{c.c.}\right)=0
\end{dcases}
\end{equation}
(where as usual the notation $\mrm{c.c.}$ denotes the complex conjugate of the term immediately before it), to be solved for $q\in\m{A}^0(K_\Sigma^2)$ and a K\"ahler form $\omega$ cohomologous to $\omega_0$, where all metric quantities are computed with respect to $\omega$. The vector field $g(\nabla^1 q,\bar{q})\diff_z$ is given by
\begin{equation*}
g(\nabla^1 q,\bar{q})\diff_z=g(\bdiff q,\bar{q})^\sharp=\left(g^{1\bar{1}}\right)^3\nabla_{\bar{1}}q_{11}\,q_{\bar{1}\bar{1}}\,\diff_z.
\end{equation*}

\subsection{The complex moment map}\label{sec:complex_mm}

Let us focus on the first equation, corresponding to the complex moment map.
\begin{lemma}
The kernel of the operator
\begin{equation*}
{\nabla^{1,0}}^*:\m{A}^{1,0}(\Sigma)\to\m{C}^\infty_0(\Sigma)
\end{equation*}
is the space $H^0(K_\Sigma)$.
\end{lemma}

\begin{proof}
Let $\beta$ be a $(1,0)$--form. Then
\begin{equation*}
{\nabla^{1,0}}^*\beta=-g^{1\bar{1}}\diff_{\bar{z}}\beta_1
\end{equation*}
so ${\nabla^{1,0}}^*\beta=0$ if and only if $\bdiff\beta=0$.
\end{proof}
So in order to solve the complex moment map equation we can simply fix a holomorphic $1$-form $\beta$ and solve
\begin{equation}\label{eq:complex_mm_curve_abeliandiff}
{\nabla^{1,0}}^*q=\beta.
\end{equation}
In equation \eqref{eq:complex_mm_curve_abeliandiff}, ${\nabla^{1,0}}^*$ is the formal adjoint of
\begin{equation*}
\nabla^{1,0}\!:\m{A}^{1,0}(\Sigma)\to\Gamma(K^2_\Sigma).
\end{equation*}
Since ${\nabla^{1,0}}^*\!:\Gamma(K^2_\Sigma)\to\Gamma(K_\Sigma)$ is an elliptic operator, by the Fredholm alternative we know that there is a solution $q$ to equation \eqref{eq:complex_mm_curve_abeliandiff} if and only if $\beta$ is orthogonal to the kernel of $\nabla^{1,0}$.

\begin{lemma}
The kernel of $\nabla^{1,0}\!:\m{A}^0(K_\Sigma)\to\m{A}^0(K_\Sigma^2)$ is trivial.
\end{lemma}

\begin{proof}
Assume that $\eta$ is in the kernel of $\nabla^{1,0}\!:\m{A}^0(K_\Sigma)\to\m{A}^0(K_\Sigma^2)$, and let $X:=\bar{\eta}^\sharp\in\Gamma(T^{1,0}\Sigma)$. Then $\nabla^{0,1}\bar{\eta}=0$, but this happens if and only if
\begin{equation*}
0=\nabla_{\bar{1}}\eta_{\bar{1}}\,\mrm{d}\bar{z}^2=g_{1\bar{1}}\,\nabla_{\bar{1}}X^1\,\mrm{d}\bar{z}^2
\end{equation*}
if and only if $X$ is holomorphic. But since $g(\Sigma)>1$ there are no nonzero holomorphic vector fields on $\Sigma$, so $\eta=0$.
\end{proof}

Hence for all fixed $\beta$ there is a solution to equation \eqref{eq:complex_mm_curve_abeliandiff}. Moreover, there is a unique solution orthogonal to the kernel of ${\nabla^{1,0}}^*$, i.e. there is a unique solution of equation \eqref{eq:complex_mm_curve_abeliandiff} that is in the image of $\nabla^{1,0}$.

\begin{lemma}
The kernel of
\begin{equation*}
{\nabla^{1,0}}^*:\Gamma(K^2_\Sigma)\to\m{A}^{1,0}(\Sigma)
\end{equation*}
is the space of holomorphic quadratic differentials.
\end{lemma}

\begin{proof}
Just compute in coordinates:
\begin{equation*}
{\nabla^{1,0}}^*q=-g^{1\bar{1}}\nabla_{\bar{1}}q_{11}\,\mrm{d}z=-g^{1\bar{1}}\diff_{\bar{z}}q_{11}\,\mrm{d}z
\end{equation*}
so ${\nabla^{1,0}}^*q=0$ if and only if $\diff_{\bar{z}}q_{11}=0$.
\end{proof}

Bringing together these facts, we deduce that for any holomorphic $1$-form $\beta$, any solution $q$ of \eqref{eq:complex_mm_curve_abeliandiff} can be written as
\begin{equation}
q=\tau+\nabla^{1,0}\eta(\beta)
\end{equation} 
where $\tau$ is a holomorphic quadratic differential and $\eta(\beta)$ is the unique $(1,0)$--form that solves
\begin{equation*}
{\nabla^{1,0}}^*\nabla^{1,0}\eta=\beta.
\end{equation*}
Of course $\eta(\beta)$ can be written as $\eta=G(\beta)$, where $G$ is the Green's operator associated to the elliptic operator ${\nabla^{1,0}}^*\nabla^{1,0}\!:\Gamma(K_\Sigma)\to\Gamma(K^2_\Sigma)$. So the set of solutions to the complex moment map equation can be written as the $(4g-3)$--dimensional complex vector space 
\begin{equation*}
\m{V}=\set*{\tau+\nabla^{1,0}G(\beta)\tc\beta\in H^0(K_\Sigma)\mbox{ and }\tau\in H^0(K^2_\Sigma)}.
\end{equation*}
The solutions to the complex moment map equation considered in \cite{Donaldson_hyperkahler} and \cite{Hodge_phd_thesis} form a codimension-$g$ vector subspace of $\m{V}$ and correspond to setting $\beta = 0$.


Let $L\!: \Gamma(K_\Sigma) \to \Gamma(K_\Sigma)$ be the self--adjoint elliptic operator defined by $L(\varphi)={\nabla^{1,0}}^*\nabla^{1,0}\varphi$. The standard Schauder estimates for elliptic operators on $\m{C}^{k,\alpha}(M,\omega)$ tell us that there is a constant $C=C(\omega,\alpha,k)$ such that
\begin{equation}\label{eq:elliptic_estimate_nablanabla}
\norm{\varphi}_{k,\alpha}\leq C\left(\norm{L\varphi}_{k-2,\alpha}+\norm{\varphi}_0\right),
\end{equation}
so for the Green operator we have
\begin{lemma}\label{lemma:stime_operatore_Green}
Let $\beta\in\m{A}^{1,0}(\Sigma)$, and let $\eta\in\m{A}^{1,0}(\Sigma)$ be the unique solution to
\begin{equation*}
{\nabla^{1,0}}^*\nabla^{1,0}\eta=\beta.
\end{equation*}
Then, for every $k\geq 2$
\begin{equation*}
\norm{\eta}_{k,\alpha}\leq K\norm{\beta}_{k-2,\alpha}
\end{equation*}
for some constant $K>0$ that does not depend on $\eta$, $\beta$.
\end{lemma}
This result is analogous to \cite[Proposition $2.3$]{KodairaMorrow}. The proof there is relative to the Green operator associated to the Laplacian, but it also goes through in our situation; the key points are an elliptic estimate, the linearity of the operator and its self-adjointness. We give a proof of Lemma \ref{lemma:stime_operatore_Green} anyway, for completeness.
\begin{proof}
Let as before $L:={\nabla^{1,0}}^*\nabla^{1,0}$. By the elliptic estimate \eqref{eq:elliptic_estimate_nablanabla} we have, for any $\varphi$
\begin{equation*}
\norm{G\varphi}_{k,\alpha}\leq C\left(\norm{\varphi}_{k-2,\alpha}+\norm{G\varphi}_0\right)
\end{equation*}
so it will be enough to show that there is a constant $C'$ such that $\norm{G\varphi}_0\leq C'\norm{\varphi}_{k-2,\alpha}$ for every $\varphi$. Assume that this is not the case. Then we can find a sequence $\varphi_n$ such that
\begin{equation*}
\frac{\norm{G\varphi_n}_0}{\norm{\varphi_n}_{k-2,\alpha}}\to\infty
\end{equation*}
so the sequence $\psi_n:=\frac{1}{\norm{G\varphi_n}_0}\varphi_n$ satisfies
\begin{equation*}
\norm{G\psi_n}_0=1\mbox{ and }\norm{\psi_n}_{k-2,\alpha}\to 0.
\end{equation*}
In particular, together with the elliptic estimate, this implies
\begin{equation*}
\norm{G\psi_n}_{k,\alpha}\leq K
\end{equation*}
for some constant $K$. By Ascoli-Arzelà Theorem we can assume that there is a $\vartheta$ such that for every $h\leq k$ we have uniform convergence $\nabla^hG\psi_n\to\nabla^h\vartheta$, up to choosing a subsequence of $\set{\psi_n}$. Then:
\begin{equation*}
\norm{\vartheta}_{L^2}^2=\lim \left\langle G\psi_n,\eta\right\rangle_{L^2}=\lim\left\langle G\psi_n,LG\vartheta\right\rangle_{L^2}=\lim\left\langle LG\psi_n,G\vartheta\right\rangle_{L^2}=\lim\left\langle\psi_n,G\vartheta\right\rangle_{L^2}=0
\end{equation*}
since $\psi_n\to 0$ in $\m{C}^{k-2,\alpha}$. But this is a contradiction: indeed $\norm{\vartheta}_0=\lim\norm{G\psi_n}_0=1$.
\end{proof}

In particular we deduce from Lemma \ref{lemma:stime_operatore_Green} that for every $\alpha\in(0,1)$, if ${\nabla^{1,0}}^*\nabla^{1,0}\eta=\beta$ then
\begin{equation*}
\norm{\nabla^{1,0}_\omega\eta}_0\leq\norm{\eta}_{2,\alpha}\leq \tilde{C}\norm{\beta}_{0,\alpha}.
\end{equation*}
So for $q=\tau+\nabla^{1,0}\eta$ we see that if for some $\alpha$ the $\m{C}^{0,\alpha}(\omega)$--norms of $\tau$, $\beta$ are small enough then we also have $\norm{q}_0^2<1$, as required by the real moment map equation.

\begin{nota}
Let us consider what happens when $g(\Sigma) \leq 1$, that is, when $\Sigma=\bb{CP}^1$ or $\Sigma=\bb{C}/\Lambda$ for a lattice $\Lambda<\bb{C}$.

In the first case $\Sigma=\bb{CP}^1$ there are no holomorphic $1$-forms or holomorphic quadratic differentials, so the only solution to the complex moment map equation is $q=0$ and the HcscK system reduces to the cscK equation.

When $\Sigma$ is a torus, if we consider systems of coordinates on $\Sigma$ induced by affine coordinates on $\bb{C}$ via the projection $\bb{C}\rightarrow\bb{C}/\Lambda$, then holomorphic objects on $\Sigma$ have constant coefficients. It is immediate then to see, by the Fredholm alternative for ${\nabla^{1,0}}^*$, that the equation ${\nabla^{1,0}}^*q=\beta$ can be solved precisely when the holomorphic form $\beta$ is $0$. In this case then $q$ must be a holomorphic quadratic differential.

Hence, by fixing an affine coordinate $z$ on the torus we see that the real moment map equation is satisfied if and only if
\begin{equation*}
\Delta\,\mrm{log}\left(g_{1\bar{1}}\left(1+\sqrt{1-(g^{1\bar{1}})^2q_{11}q_{\bar{1}\bar{1}}}\right)\right)=0
\end{equation*}
since $g_{1\bar{1}}$ can be regarded as a (global) positive function on $\Sigma$ and $q_{11}$ is a constant. But then $g_{1\bar{1}}\left(1+\sqrt{1-(g^{1\bar{1}})^2q_{11}q_{\bar{1}\bar{1}}}\right)$ must be a constant, and this happens only if $g$ is the flat metric in its class. So, even for $g(\Sigma)=1$, the HcscK equations essentially reduce to the cscK equation.
\end{nota}

\subsection{A change of variables}

The upshot of the previous section is that a unique solution $q$ to the complex moment map equation can always be found, for a fixed metric $\omega$, by prescribing two parameters $\tau\in H^0(K_\Sigma^2)$, $\beta\in H^0(K_\Sigma)$. The corresponding $q(\omega,\tau,\beta)$ is given by
\begin{equation*}
q=\tau+\nabla^{1,0}\eta(\beta)
\end{equation*}
where $\eta(\beta)$ is the unique solution of ${\nabla^{1,0}}^*\nabla^{1,0}\eta=\beta$, so our system becomes
\begin{equation}\label{eq:sistema_HCSCK_beta}
\begin{dcases}
{\nabla^{1,0}}^*q=\beta;\\
2\,s(\omega)-2\,\widehat{s}+\Delta\left(\mrm{log}\left(1+\sqrt{1-\norm{q}^2}\right)\right)-\mrm{div}\left(\frac{\bar{q}\left(-,\beta^\sharp\right)^\sharp}{1+\sqrt{1-\norm{q}^2}}+\mrm{c.c.}\right)=0;\\
\norm{q}^2_\omega<1.
\end{dcases}
\end{equation}

In order to study the real moment map equation we take an approach analogous to the one in \cite{Donaldson_hyperkahler}, by performing a suitable change of variables. 

Let $F:=1+\sqrt{1-\norm{q}^2_\omega}$, and consider the K\"ahler form $\tilde{\omega}:=F\,\omega$. Notice that $\omega$ can be recovered from $\tilde{\omega}$ and $q$, by $\omega=\frac{1}{2}\left(1+\norm{q}^2_{\tilde{\omega}}\right)\tilde{\omega}$. Indeed, a quick computation shows that
\begin{equation*}
\frac{2}{1+\sqrt{1-\norm{q}^2_\omega}}=1+\frac{\norm{q}^2_\omega}{\left(1+\sqrt{1-\norm{q}^2_\omega}\right)^2}=1+\norm{q}^2_{\tilde{\omega}}
\end{equation*}
so that $F^{-1}=\frac{1}{2}(1+\norm{q}^2_{\tilde{\omega}})$. We also have the following identities:
\begin{equation*}
s(\omega)=F\,s(\tilde{\omega})-\frac{1}{2}\Delta(\mrm{log}\,F);
\end{equation*}
\begin{equation*}
\begin{split}
\mrm{div}_\omega&\left(\frac{\bar{q}\left(-,\beta^\sharp\right)^\sharp}{F}+\mrm{c.c.}\right)=\mrm{div}_\omega\left(F\,\bar{q}\left(-,\beta^{\tilde{\sharp}}\right)^{\tilde{\sharp}}+\mrm{c.c.}\right)=\\
&=\mrm{div}_{\tilde{\omega}}\left(F\,\bar{q}\left(-,\beta^{\tilde{\sharp}}\right)^{\tilde{\sharp}}+\mrm{c.c.}\right)-\frac{1}{F}\left(F\,\bar{q}\left(-,\beta^{\tilde{\sharp}}\right)^{\tilde{\sharp}}(F)+\mrm{c.c.}\right)=\\
&=F\,\mrm{div}_{\tilde{\omega}}\left(\bar{q}\left(-,\beta^{\tilde{\sharp}}\right)^{\tilde{\sharp}}+\mrm{c.c.}\right)
\end{split}
\end{equation*}
so that $\omega$ solves the second equation in \eqref{eq:sistema_HCSCK_beta} if and only if $\tilde{\omega}$ solves
\begin{equation*}
2\,s(\tilde{\omega})-\frac{2}{F}\widehat{s}-\mrm{div}_{\tilde{\omega}}\left(\bar{q}\left(-,\beta^{\tilde{\sharp}}\right)^{\tilde{\sharp}}+\mrm{c.c.}\right)=0,
\end{equation*}
if and only if
\begin{equation*}
2\,s(\tilde{\omega})-\widehat{s}\left(1+\norm{q}^2_{\tilde{\omega}}\right)-\mrm{div}_{\tilde{\omega}}\left(\bar{q}\left(-,\beta^{\tilde{\sharp}}\right)^{\tilde{\sharp}}+\mrm{c.c.}\right)=0.
\end{equation*}
These computations show that $\omega,q$ solve the HcscK system if and only if $\tilde{\omega},q$ solve
\begin{equation}\label{eq:HcscK_trasformata}
\begin{dcases}
\frac{2}{1+\norm{q}^2_{\tilde{\omega}}}{\nabla^{1,0}_{\tilde{\omega}}}^*q=\beta;\\
2\,s(\tilde{\omega})-\widehat{s}\left(1+\norm{q}^2_{\tilde{\omega}}\right)-\mrm{div}_{\tilde{\omega}}\left(\bar{q}\left(-,\beta^{\tilde{\sharp}}\right)^{\tilde{\sharp}}+\mrm{c.c.}\right)=0;\\
\norm{q}^2_{\tilde{\omega}}<1.
\end{dcases}
\end{equation}
We can use the first equation in \eqref{eq:HcscK_trasformata} to write the second one as
\begin{equation*}
2\,s(\tilde{\omega})-\widehat{s}\left(1+\norm{q}^2_{\tilde{\omega}}\right)
-\left(\tilde{g}(\bar{q},\nabla^{1,0}\beta)+\mrm{c.c.}\right)+(1+\norm{q}^2_{\tilde{\omega}})\norm{\beta}^2_{\tilde{\omega}}=0
\end{equation*}
or equivalently, after a little simplification, 
\begin{equation}
2\,s(\tilde{\omega})+\left(-\widehat{s}+\norm{\beta}^2_{\tilde{\omega}}\right)\left(1+\norm{q}^2_{\tilde{\omega}}\right)
-\left(\tilde{g}(\bar{q},\nabla^{1,0}\beta)+\mrm{c.c.}\right)=0.
\end{equation}

\subsection{The equations for a conformal potential}

In order to solve our equations \eqref{eq:HcscK_trasformata} we take the standard approach of fixing a reference K\"ahler form, still denoted by $\tilde{\omega}$, and of looking for solutions in its conformal class, that is, of the form $e^{f}\tilde{\omega}$. A straightforward computation shows that our equations written in terms of the unknown $f$ become
\begin{equation}\label{eq:HcscK_trasf_conf_pot}
\begin{dcases}
\frac{2\,\mrm{e}^{-f}}{1+\norm{\mrm{e}^{-f}q}^2_{\tilde{\omega}}}{\nabla^{1,0}_{\tilde{\omega}}}^*q=\beta;\\
2\,s(\tilde{\omega})+\Delta_{\tilde{\omega}}f+\left(-\mrm{e}^f\widehat{s}+\norm{\beta}^2_{\tilde{\omega}}\right)\left(1+\norm{\mrm{e}^{-f}q}^2_{\tilde{\omega}}\right)-\left(\tilde{g}(\mrm{e}^{-f}\bar{q},\nabla^{1,0}_{\tilde{\omega}}\beta-\beta\otimes\diff f)+\mrm{c.c.}\right)=0;\\
\norm{\mrm{e}^{-f}q}^2_{\tilde{\omega}}<1.
\end{dcases}
\end{equation}
Here $\widehat{s}=\widehat{s(\omega)}$ is still computed using the original metric $\omega$.

Of course we may also do things in the opposite order: we can first write our original system \eqref{eq:HcscK_system_curve} in terms of a conformal factor and then perform the change of independent variables described in the previous section. In fact this yields the same equations \eqref{eq:HcscK_trasf_conf_pot}. To see this write \eqref{eq:HcscK_system_curve} in terms of a reference K\"ahler form, still denoted by $\omega$, and a conformal metric $\omega_f=\mrm{e}^f\omega$, giving   
\begin{equation}\label{eq:HcscK_system_curve_confpot}
\begin{dcases}
{\nabla^{1,0}_{\omega}}^*q=\mrm{e}^f\beta;\\
2\,s(\omega)+\Delta_\omega(f)-2\,\mrm{e}^f\widehat{s(\omega)}+\Delta_{\omega}\left(\mrm{log}\left(1+\sqrt{1-\norm{\mrm{e}^{-f}q}_{\omega}^2}\right)\right)-\mrm{div}_{\omega}\left(\frac{\mrm{e}^{-f}\,\bar{q}\left(-,\beta^\sharp\right)^\sharp}{1+\sqrt{1-\norm{\mrm{e}^{-f}q}_{\omega}^2}}+\mrm{c.c.}\right)=0;\\
\norm{\mrm{e}^{-f}q}^2_{\omega}<1.
\end{dcases}
\end{equation}
Notice first of all that if $\omega_f$ satisfies the second equation in \eqref{eq:HcscK_system_curve_confpot} then $\omega_f$ is necessarily in the same K\"ahler class of $\omega$, since the constant which appears is $\widehat{s(\omega)}$ rather than $\widehat{s(\omega_f)}$. Now we can rewrite this system in terms of $\omega' =\left(1+\sqrt{1-\norm{\mrm{e}^{-f}q}^2_\omega}\right)\omega$ and a computation shows that this is the same as \eqref{eq:HcscK_trasf_conf_pot}, with $\tilde{\omega}$ replaced by $\omega'$.

The upshot of this observation is that there is a bijection between the solutions of \eqref{eq:HcscK_trasf_conf_pot} and those of \eqref{eq:HcscK_system_curve_confpot}, given by mapping $(q, \mrm{e}^f \tilde{\omega})$ to $(q, \mrm{e}^f \omega)$, and a solution $\mrm{e}^f\omega$ is automatically cohomologous to the original metric $\omega$. In particular the ``complex moment map" equation in \eqref{eq:HcscK_trasf_conf_pot}, that is
\begin{equation}\label{eq:trasf_complex_mm_confpot}
\frac{2\,\mrm{e}^{-f}}{1+\norm{\mrm{e}^{-f}q}^2_{\tilde{\omega}}}{\nabla^{1,0}_{\tilde{\omega}}}^*q=\beta
\end{equation}
is equivalent to $\mrm{e}^{-f}{\nabla^{1,0}_\omega}^*q=\beta$, which we already solved in Section \ref{sec:complex_mm}. 

\section{A continuity method}\label{ContinuitySec}

In the previous Section we showed that the original HcscK system is equivalent to \eqref{eq:HcscK_trasf_conf_pot}. We will solve this system, under appropriate conditions on $\tau$ and $\beta$, by using a continuity method. 

It is convenient to change our notation for the background metric, appearing in \eqref{eq:HcscK_trasf_conf_pot}, denoting it simply by $\omega$. We take the background metric $\omega$ to have constant negative Gauss curvature. Without loss of generality we can normalize $\omega$ so that the constant $\widehat{s}$ in \eqref{eq:HcscK_trasf_conf_pot} is equal to $-2$, and we consider the family of equations \eqref{eq:HcscK_continuity} parametrized by $t\in[0,1]$,
\begin{equation}
\tag{$\star_t$}\label{eq:HcscK_continuity}
\begin{dcases}
-2+\Delta f_t+\left(2\,\mrm{e}^{f_t}+\norm{t\beta}^2\right)\left(1+\mrm{e}^{-2\,f_t}\norm{q_t}^2\right)-\left(g(\mrm{e}^{-f_t}\bar{q_t},\nabla^{1,0}\left(t\beta\right)-(t\beta)\otimes\diff f_t)+\mrm{c.c.}\right)=0\\
\mrm{e}^{-2\,f_t}\norm{q_t}^2<1
\end{dcases}
\end{equation}
where $q_t =t\tau+\nabla^{1,0}\eta_t$ solves
\begin{equation*}
\frac{2\,\mrm{e}^{-f_t}}{1+\mrm{e}^{-2f_t}\norm{q_t}^2}{\nabla^{1,0}}^*q_t=t\,\beta.
\end{equation*}
Here all metric quantities are computed with respect to the background $\omega$, as usual.

For $t=0$ we have the solution $f\equiv 0$ to $(\star_0)$, and we propose to show that, under some boundedness assumptions of $\tau$, $\beta$, $\nabla\beta$, we can find a solution $f$ to $(\star_1)$. To prove closedness of the continuity method we need \textit{a priori} $\m{C}^{k,\alpha}$-estimates on $f_t$ and $q_t$, for some $k\geq 2$ and some $0<\alpha<1$. Moreover, crucially, we also need to show that the open condition $\mrm{e}^{-2\,f_t}\norm{q_t}^2<1$ is also closed.

As a preliminary step we first establish such estimates on the quadratic differential $q$, along the continuity path, in terms of given H\"older bounds on $\tau$, $\beta$, and a H\"older bound on $f$. The latter will be then proved in the following sections. In what follows all metric quantities are computed with respect to $\omega$. We already showed that a solution $q$ to \eqref{eq:trasf_complex_mm_confpot} can be decomposed as $q=\tau+\nabla^{1,0} \eta$ for some $\eta\in\m{A}^{1,0}(\Sigma)$ and $\tau\in H^0(K^2_\Sigma)=\mrm{ker}({\nabla^{1,0}}^*)$. Thus $\eta$ solves the equation
\begin{equation*}
\frac{2\,\mrm{e}^{-f}}{1+\mrm{e}^{-2f}\norm{\tau+\nabla^{1,0} \eta}^2 }{\nabla^{1,0} }^*\nabla^{1,0} \eta=\beta.
\end{equation*}
We write this in the form
\begin{equation*}
{\nabla^{1,0} }^* \nabla^{1,0}  \eta=\frac{1}{2}\beta\left(\mrm{e}^f+\mrm{e}^{-f}\norm{\tau+\nabla^{1,0} \eta}^2 \right)
\end{equation*}
and use the standard estimate given in Lemma \ref{lemma:stime_operatore_Green} to show that for all $k\geq 2$ and $\alpha\in(0,1)$ there are constants $C, C' > 0$ such that
\begin{equation*}
\begin{split}
\norm{\eta}_{k,\alpha}&\leq C\norm{\beta}_{k-2,\alpha}\left(\norm{\mrm{e}^f}_{k-2,\alpha}+\norm{\mrm{e}^{-f}}_{k-2,\alpha}\norm*{\norm{\tau+\nabla^{1,0} \eta}^2 }_{k-2,\alpha}\right) \\
&\leq C\mrm{e}^{\norm{f}_{k-2,\alpha}}\norm{\beta}_{k-2,\alpha}\left(1+C'\norm*{\tau+\nabla^{1,0}\eta}^2_{k-2,\alpha}\right) \\
&\leq C\mrm{e}^{\norm{f}_{k-2,\alpha}}\norm{\beta}_{k-2,\alpha}\left(1+C'\left(\norm{\tau}^2_{k-2,\alpha}+\norm{\eta}^2_{k,\alpha}+2\norm{\tau}_{k-2,\alpha}\norm{\eta}_{k,\alpha}\right)\right).
\end{split}
\end{equation*}
Note that going from the first to the second inequality involves a short computation using that $\nabla$ is the Levi-Civita connection. In this estimate only the constant $C$ depends on $\omega$, and the dependence is only through the elliptic constant $K$ appearing in Lemma \ref{lemma:stime_operatore_Green}. We can rewrite this inequality in the form
\begin{equation*}
\norm{\eta}_{k,\alpha}\leq c+b\norm{\eta}_{k,\alpha}+a\norm{\eta}^2_{k,\alpha}
\end{equation*}
where $a$, $b$ and $c$ are functions of $\norm{f}_{k-2,\alpha}$, $\norm{\beta}_{k-2,\alpha}$ and $\norm{\tau}_{k-2,\alpha}$, which can be made explicit in terms of the elliptic constant $K$, and become arbitrarily small if $\norm{\beta}_{k-2,\alpha}$ is small enough, depending on $K$. So if $\norm{f}_{k-2,\alpha}$, $\norm{\beta}_{k-2,\alpha}$ and $\norm{\tau}_{k-2,\alpha}$ satisfy a suitable bound, which only depends on $\omega$ through $K$, then we have $1-b>0$ and $(1-b)^2-4ac>0$, and we find
\begin{equation*}
0\leq\norm{\eta}_{k,\alpha}\leq\frac{1-b-\sqrt{(1-b)^2-4ac}}{2a}\quad\mbox{ or }\quad\norm{\eta}_{k,\alpha}\geq\frac{1-b+\sqrt{(1-b)^2-4ac}}{2a}.
\end{equation*}
Since for $\beta=0$ the only solution to our equation is $\eta=0$, along the continuity path \eqref{eq:HcscK_continuity} we obtain the  bounds
\begin{equation}\label{eq:stime_eta_cambiamentocoordinate_sintetica}
\norm{\eta}_{k,\alpha}\leq\frac{1-b-\sqrt{(1-b)^2-4ac}}{2a}.
\end{equation}
In particular, for $k=2$ we get bounds on $\norm{\eta}_0$ in terms of the $\m{C}^{0,\alpha}$-norms of $\beta$, $\tau$, $f$. The bound \eqref{eq:stime_eta_cambiamentocoordinate_sintetica} on $\eta$ may be written more explicitly as
\begin{equation}\label{eq:stime_eta_cambiamentocoordinate}
\norm{\eta}_{k,\alpha}\leq C\mrm{e}^{\norm{f}_{k-2,\alpha}}\norm{\beta}_{k-2,\alpha}+O(\norm{\beta}_{k-2,\alpha}\norm{\tau}_{k-2,\alpha}),
\end{equation}
(where the $O$ term depends on the background $\omega$ only through the constant $K$), and holds as long as
\begin{equation*}
C\,C'\,\mrm{e}^{\norm{f}_{k-2,\alpha}}\norm{\beta}_{k-2,\alpha}\norm{\tau}_{k-2,\alpha}<1
\end{equation*}
and
\begin{equation*}
C\,C'\,\mrm{e}^{\norm{f}_{k-2,\alpha}}\norm{\beta}_{k-2,\alpha}\left(\norm{\beta}_{k-2,\alpha}  C\,\mrm{e}^{\norm{f}_{k-2,\alpha}} \left(3 C' \norm{\tau}_{k-2,\alpha}^2+4\right)+2\norm{\tau}_{k-2,\alpha} \right)<1.
\end{equation*}

\subsection{Estimates along the continuity method}\label{EstimatesSec}
We now proceed to establish H\"older bounds on the conformal potential $f$.
\subsubsection{$\m{C}^0$-estimates}

Let $(q,f)$ be a solution to \eqref{eq:HcscK_trasf_conf_pot}. Then, at a point at which $f$ attains its maximum we have 
\begin{equation*}
-2+\left(2\,\mrm{e}^f+\norm{\beta}^2\right)\left(1+\mrm{e}^{-2\,f}\norm{q}^2\right)-2\,\mrm{Re}\left(g\left(\mrm{e}^{-f}\bar{q},\nabla^{1,0}\beta\right)\right)\leq 0
\end{equation*}
(recall that our convention is $\Delta=-\mrm{div}\,\mrm{grad}$, so that $\Delta(f)$ is \emph{positive} where $f$ attains its maximum). As we are assuming $\mrm{e}^{-2\,f}\norm{q}^2<1$, by the Cauchy--Schwarz inequality we have
\begin{equation*}
\card*{\mrm{Re}\left(g\left(\mrm{e}^{-f}\bar{q},\nabla^{1,0}\beta\right)\right)}\leq\norm{\nabla^{1,0}\beta}
\end{equation*}
and so, at a maximum of $f$
\begin{equation*}
0\geq -2+\left(2\,\mrm{e}^f+\norm{\beta}^2\right)\left(1+\mrm{e}^{-2\,f}\norm{q}^2\right)-2\,\mrm{Re}\left(g\left(\mrm{e}^{-f}\bar{q},\nabla^{1,0}\beta\right)\right)\geq  -2+\left(2\,\mrm{e}^f+\norm{\beta}^2\right)-2\norm{\nabla^{1,0}\beta}
\end{equation*}
hence we find that
\begin{equation*}
\mrm{e}^f\leq 1+\norm{\nabla^{1,0}\beta}.
\end{equation*}

Similarly, at a point of minimum of $f$ we find
\begin{equation*}
-2+\left(2\,\mrm{e}^f+\norm{\beta}^2\right)\left(1+\mrm{e}^{-2\,f}\norm{q}^2\right)-2\,\mrm{Re}\left(g\left(\mrm{e}^{-f}\bar{q},\nabla^{1,0}\beta\right)\right)\geq 0.
\end{equation*}
The same estimates then imply
\begin{equation*}
0\leq-2+\left(2\,\mrm{e}^f+\norm{\beta}^2\right)\left(1+\mrm{e}^{-2\,f}\norm{q}^2\right)-2\,\mrm{Re}\left(g\left(\mrm{e}^{-f}\bar{q},\nabla^{1,0}\beta\right)\right)\leq -2+2\left(2\,\mrm{e}^f+\norm{\beta}^2\right)+2\,\norm{\nabla^{1,0}\beta}
\end{equation*}
so that
\begin{equation*}
2\,\mrm{e}^f\geq 1-\norm{\nabla^{1,0}\beta}-\norm{\beta}^2.
\end{equation*}
If $\beta$ is chosen in such a way that $\norm{\nabla^{1,0}\beta}+\norm{\beta}^2\leq 1-2\varepsilon$ then $\mrm{e}^f$ is uniformly bounded away from $0$ by $\varepsilon$, and we have a $\m{C}^0$-bound for solutions of $(\star_1)$ (and similarly for solutions of any \eqref{eq:HcscK_continuity}).

\subsubsection{$L^4$-bounds on the gradient and the Laplacian}

Our $\m{C}^0$-bound on $f$ can be used to obtain an estimate for the $L^2$-norm of $\mrm{d}f$. Since $f$ solves
\begin{equation*}
\Delta(f)-2+\left(2\,\mrm{e}^f+\norm{\beta}^2\right)\left(1+\mrm{e}^{-2\,f}\norm{q}^2\right)-2\,\mrm{Re}\left(g\left(\mrm{e}^{-f}\bar{q},\nabla^{1,0}\beta\right)\right)+2\,\mrm{Re}\left(g\left(\mrm{e}^{-f}\bar{q},\beta\otimes\diff f\right)\right)=0,
\end{equation*}
the identity
\begin{equation*}
\int_{\Sigma}\norm{\mrm{d}f}^2\omega=\int_{\Sigma} f\Delta(f)\,\omega
\end{equation*}
shows that we have
\begin{equation*}
\norm{\mrm{d}f}^2_{L^2}=\int f\left[2-\left(2\,\mrm{e}^f+\norm{\beta}^2\right)\left(1+\mrm{e}^{-2\,f}\norm{q}^2\right)+2\,\mrm{Re}\left(g\left(\mrm{e}^{-f}\bar{q},\nabla^{1,0}\beta\right)\right)-2\,\mrm{Re}\left(g\left(\mrm{e}^{-f}\bar{q},\beta\otimes\diff f\right)\right)\right]\omega.
\end{equation*}
Expanding out the product in the integrand, we see that the first three terms can be bounded explicitly in terms of $\norm{\beta}_{0}$ and $\norm{\nabla\beta}_{0}$ using the $\m{C}^0$-bound on $f$. As for the last term, we have by Cauchy--Schwarz
\begin{equation*}
\begin{split}
\int f\,2\,\mrm{Re}\left(g\left(\mrm{e}^{-f}\bar{q},\beta\otimes\diff f\right)\right)\omega=&\left\langle \beta\otimes\diff f,f\,\mrm{e}^{-f}q\right\rangle_{L^2}+\mrm{c.c.}\leq 2\,\norm*{f\,\mrm{e}^{-f}q}_{L^2}\norm*{\beta}_{L^2}\norm*{\diff f}_{L^2}<\\
<&\sqrt{2}\,\norm{f}_{L^2}\norm*{\beta}_{L^2}\norm*{\mrm{d}f}_{L^2}.
\end{split}
\end{equation*}
So there are some positive constants $C_1$ and $C_2$ that depend explicitly on our $\m{C}^0$-bound for $f$ and a bound for $\norm{\beta}_0$, such that
\begin{equation*}
\norm*{\mrm{d}f}^2_{L^2}<C_1+C_2\norm*{\mrm{d}f}_{L^2},
\end{equation*}
which clearly gives a bound on the $L^2$-norm of $\mrm{d}f$.

Now we write our equation as
\begin{equation*}
\Delta(f)=2-\left(2\,\mrm{e}^f+\norm{\beta}^2\right)\left(1+\mrm{e}^{-2\,f}\norm{q}^2\right)+2\,\mrm{Re}\left(g\left(\mrm{e}^{-f}\bar{q},\nabla^{1,0}\beta\right)\right)-2\,\mrm{Re}\left(g\left(\mrm{e}^{-f}\bar{q},\beta\otimes\diff f\right)\right).
\end{equation*}
Using the $\m{C}^0$-estimate, the condition $\mrm{e}^{-f}\norm{q} <1$ and the Cauchy--Schwarz inequality we get
\begin{equation*}
\card*{\Delta(f)}\leq C_3+2\card*{\mrm{Re}\left(g\left(\mrm{e}^{-f}\bar{q},\beta\otimes\diff f\right)\right)}\leq C_3+C_4\norm{\mrm{d}f} 
\end{equation*}
for positive constants $C_3$, $C_4$ that depend on $\norm{\beta}_0$, $\norm{\nabla^{1,0}\beta}_0$ and the $\m{C}^0$--estimate on $f$. This implies
\begin{equation*}
\norm{\Delta f}_{L^2}\leq\norm*{C_3+C_4\norm{\mrm{d}f}_\omega}_{L^2}\leq C_3+C_4\norm{\mrm{d}f}_{L^2}
\end{equation*}
so the $L^2$-bound on $\mrm{d}f$ gives us a $L^2$-bound on $\Delta f$. The same reasoning actually shows that $L^p$-bounds on $\mrm{d}f$ will imply $L^p$-bounds on $\Delta f$.

Recall the Sobolev inequality
\begin{equation*}
\norm{u}_{L^2}\leq K_1\norm{u}_{W^{1,1}}.
\end{equation*}
In particular for $u=\norm{\mrm{d}f}^2_\omega$ we find
\begin{equation*}
\norm*{\norm{\mrm{d}f}^2_\omega}_{L^2}\leq K_1\norm*{\norm{\mrm{d}f}^2_\omega}_{W^{1,1}}= K_1\left(\norm*{\norm{\mrm{d}f}^2_\omega}_{L^1}+\norm*{\nabla\norm{\mrm{d}f}^2_\omega}_{L^1}\right)
\end{equation*}
Now, $\norm*{\norm{\mrm{d}f}^2_\omega}_{L^1}=\norm{\mrm{d}f}^2_{L^2}$ and $\nabla\norm{\mrm{d}f}^2_\omega=2\,g(\nabla\mrm{d}f,\mrm{d}f)$, so by Cauchy--Schwarz
\begin{equation*}
\begin{split}
\norm*{\nabla\norm{\mrm{d}f}^2_\omega}_{L^1}=&\int\norm*{2\,g(\nabla\mrm{d}f,\mrm{d}f)}_\omega\omega=2\int\norm*{\nabla\mrm{d}f}_\omega\norm*{\mrm{d}f}_\omega\omega=2\left\langle\norm{\nabla\mrm{d}f}_{\omega},\norm{\mrm{d}f}_\omega\right\rangle_{L^2}\leq\\
\leq&2\norm{\nabla\mrm{d}f}_{L^2}\norm{\mrm{d}f}_{L^2}.
\end{split}
\end{equation*}
By elliptic estimates (c.f. \cite[Theorem $5.2$]{SpinGeometry_book}) we have
\begin{equation*}
\norm{\nabla d f}_{L^2}\leq K_2\left(\norm{f}_{L^2}+\norm{\Delta f}_{L^2}\right).
\end{equation*}
Thus we find
\begin{equation*}
\norm*{\norm{\mrm{d}f}^2_\omega}_{L^2}\leq K_1\left(\norm{\mrm{d}f}^2_{L^2}+2\,K_2\,\norm{\mrm{d}f}_{L^2}\left(\norm{f}_{L^2}+\norm{\Delta f}_{L^2}\right)\right).
\end{equation*}
Since $\norm*{\norm{\mrm{d}f}^2_\omega}_{L^2}=\norm{\mrm{d}f}^2_{L^4}$, from the $L^2$-bound on $\mrm{d}f$ and $\Delta(f)$ that we already have we deduce an $L^4$-bound on $\mrm{d}f$. 

Our previous discussion then shows that we can actually obtain (explicit) $L^4$-bounds on $\Delta f$, in terms of $\norm{\beta}_0$, $\norm{\nabla^{1,0}\beta}_0$, the Sobolev constant $K_1$ and the elliptic constant $K_2$.

\subsubsection{$\m{C}^{k,\alpha}$-bounds}

Recall Morrey's inequality for $n=2$, $p=4$ (c.f. \cite[\S$5.6.2$]{Evans_elliptic_PDE}):
\begin{equation*}
\norm{f}_{ 0,\frac{1}{2} }\leq K_3\norm{f}_{W^{1,4}}.
\end{equation*}
By our $L^4$ bound on $df$ this implies a $\m{C}^{0,\frac{1}{2}}$-estimate on $f$ in terms of the $\m{C}^0$-estimate on $f$, the Sobolev constants $K_1, K_3$ and the elliptic constant $K_2$. 

Moreover, the Sobolev inequality for $n=2$, $p=4$ (c.f. \cite[\S$5.6.3$]{Evans_elliptic_PDE}) tells us that
\begin{equation*}
\norm{f}_{ 3,\frac{1}{2} }\leq K_4\norm{f}_{W^{2,4}}
\end{equation*}
so our previous $L^4$ bound on $\Delta f$ gives  \emph{a priori} estimates for the $\m{C}^{3,\frac{1}{2}}$-norm of $f$ solving $(\star_1)$ (or \eqref{eq:HcscK_continuity} substituting $t\beta$ to $\beta$ in the previous discussion).

\subsection{Closedness}\label{ClosednessSec}

We can now complete the proof of closedness for our continuity path.

Our $\m{C}^{3,\frac{1}{2}}$-estimate for $f$ is enough to pass to the limit as $t \to \bar{t} \leq 1$ in the equation
\begin{equation*}
\Delta(f_t)-2+\left(2\,\mrm{e}^{f_t}+\norm{t\beta}^2\right)\left(1+\mrm{e}^{-2\,f_t}\norm{q_t}^2\right)-2\,\mrm{Re}\left(g\left(\mrm{e}^{-f_t}\bar{q_t},\nabla^{1,0}t\beta\right)\right)+2\,\mrm{Re}\left(g\left(\mrm{e}^{-f_t}\bar{q_t},t\beta\otimes\diff f_t\right)\right)=0
\end{equation*}
for $q_t = q(t\tau,f_t,t\beta)$. Bootstrapping then shows that the set of $t\in[0,1]$ for which this equation has a smooth solution is closed. Moreover, the $\m{C}^{3,\frac{1}{2}}$-estimate for $f$ follows from the $\m{C}^{0}$-estimate, which only requires the assumption $||\beta||_1 < 1$. 

What remains to be checked is that the quantity $||\mrm{e}^{-f_t}q_t||_0$ stays uniformly bounded away from $1$ along the continuity path. This is where the more refined control on the growth of $\norm{f_t}_{ 0,\frac{1}{2} }$ is required.

Our estimate \eqref{eq:stime_eta_cambiamentocoordinate} on $\eta$ for $k = 2$, $\alpha = 1/2$ immediately gives a bound on $q$ of the form \begin{equation*}
\begin{split}
\norm{q}_0&\leq  \norm{\tau}_0+\norm{\nabla^{1,0}\eta }_0\leq  \norm{\tau}_0+\norm{\eta }_{2,\frac{1}{2}} \\
&\leq  \norm{\tau}_0+C\mrm{e}^{\norm{f}_{0,\frac{1}{2}}} \norm{\beta}_{0,\frac{1}{2}}+O( \norm{\beta}_{0,\frac{1}{2}}\norm{\tau}_{0,\frac{1}{2}}).
\end{split}
\end{equation*}
Here the $O$ term depends on the background $\omega$ only through the elliptic constant $K$, and the inequality holds provided $\norm{f}_{0,\frac{1}{2}}$, $\norm{\beta}_{0,\frac{1}{2}}$, $\norm{\tau}_{0,\frac{1}{2}}$ are sufficiently small, also in terms of $K$. But we showed that there is a uniform a priori bound on $\norm{f}_{0,\frac{1}{2}}$, depending only on the condition $||\beta||_1 < 1$, the Sobolev constants $K_1, K_3$ and the elliptic constant $K_2$.

It follows that we if choose $\norm{\tau}_{0,\frac{1}{2}}$, $||\beta||_1$ small enough, depending only on the Sobolev constants $K_1, K_3$ and the elliptic constants $K, K_2$, then we can make sure that for all $t \in [0, 1]$ the norm $\norm{q_t}_0$ is sufficiently small so that the required bound
\begin{equation*}
\norm{\mrm{e}^{-f_t} q_t}_0 \leq  \mrm{e}^{\norm{f_t}_{0, \frac{1}{2}}} \norm{q_t}_{0} < 1
\end{equation*}
holds uniformly.
\subsection{Openness}\label{OpennessSec}

We complete our analysis of the continuity path \eqref{eq:HcscK_continuity} by showing that the set of times $t \in [0,1]$ for which there is a smooth solution is open. We will see that openness requires control of a further elliptic constant, namely the $\m{C}^{2,\frac{1}{2}}$ Schauder estimate for the Riemannian laplacian of the hyperbolic metric acting on functions.

It is convenient to write our equations in the form 
\begin{equation}\label{eq:sistema_openness_f}
\begin{dcases}
\frac{2}{1+\norm{q}_f^2}{\nabla^{1,0}_f}^*q=\beta\\
-2\,\mrm{e}^{-f}+\Delta_f(f)+(1+\norm{q}^2_f)(2+\norm{\beta}^2_f)-\left(g_f(\bar{q},\nabla^{1,0}_f\beta)+\mrm{c.c.}\right)=0\\
 \norm{q}^2_f<1.
\end{dcases}
\end{equation}
where the notation underlines that metric quantities are now computed with respect to the metric $\omega_f$. The last condition is clearly open, so we focus on the first two equations. These can be regarded as the zero-locus equations for the functional
\begin{equation*}
\begin{split}
\m{F}:H^0(K^2_\Sigma)\times H^0(K_\Sigma)\times\m{A}^{0}(K_\Sigma)\times\m{C}^\infty(\Sigma,\bb{R})&\to\m{A}^0(K_\Sigma)\times\m{C}^\infty(\Sigma,\bb{R})\\
(\tau,\beta,\eta,f)&\mapsto(\m{F}^1(\tau,\beta,\eta,f),\m{F}^2(\tau,\beta,\eta,f))
\end{split}
\end{equation*}
defined as
\begin{equation*}
\begin{split}
\m{F}^1(\tau,\beta,\eta,f)=&\frac{2}{1+\norm*{\tau+\nabla^{1,0}_f\eta}_f^2}{\nabla^{1,0}_f}^*\nabla^{1,0}_f\eta-\beta\\
\m{F}^2(\tau,\beta,\eta,f)=&-2\,\mrm{e}^{-f}+\Delta_f(f)+\left(1+\norm*{\tau+\nabla^{1,0}_f\eta}_f^2\right)(2+\norm{\beta}^2_f)-2\,\mrm{Re}\left(g_f\left(\bar{\tau}+\nabla^{0,1}_f\bar{\eta},\nabla^{1,0}_f\beta\right)\right).
\end{split}
\end{equation*}
Assume that $\m{F}(\tau,\beta,\eta,f)=(0,0)$. We want to show that if $\tau',\beta'$ are close enough to $\tau,\beta$ then we can also find $\eta',f'$ such that $\m{F}(\tau',\beta',\eta',f')=(0,0)$. To use the Implicit Function Theorem we should show that
\begin{equation*}
\begin{split}
\m{A}^0(K_\Sigma)\times\m{C}^\infty(\Sigma,\bb{R})&\to\m{A}^0(K_\Sigma)\times\m{C}^\infty(\Sigma,\bb{R})\\
(\dot{\eta},\varphi)&\mapsto D\m{F}_{(\tau,\beta,\eta,f)}(0,0,\dot{\eta},\varphi)
\end{split}
\end{equation*}
is surjective (on some appropriate Banach subspaces). We will show that in fact it is an isomorphism. For the rest of this section we will compute all metric quantities with respect to $\omega_f$, unless we specify otherwise, so we will drop the subscript $f$. As usual we write $q=\tau+\nabla^{1,0} \eta$.

Using $\m{F}^1(\tau,\beta,\eta,f)=0$, we compute
\begin{equation*}
\begin{split}
D\m{F}^1(\dot{\eta},\varphi)=&-\frac{\beta}{1+\norm{q}^2}\left(-2\,\varphi\,\norm{q}^2+2\,\mrm{Re}\left\langle\nabla^{1,0}_0\dot{\eta},\bar{q}\right\rangle\right)-\varphi\,\beta+\frac{2\,{\nabla^{1,0}}^*\nabla^{1,0} \dot{\eta}}{1+\norm{q}^2} \\
=&\varphi\,\beta\frac{\norm{q}^2-1}{1+\norm{q}^2}+\frac{2\,{\nabla^{1,0}}^*\nabla^{1,0}_0\dot{\eta}}{1+\norm{q}^2}-\frac{\beta}{1+\norm{q}^2}2\,\mrm{Re}\left\langle\nabla^{1,0}\dot{\eta},\bar{q}\right\rangle.
\end{split}
\end{equation*}
Similarly, using $\m{F}^2(\tau,\beta,\eta,f)=0$, we compute
\begin{equation*}
\begin{split}
D\m{F}^2(\dot{\eta},\varphi)=&\Delta(\varphi)+2\,\varphi\left(1-\norm{q}^2(1+\norm{\beta}^2)+2\mrm{Re}\langle\bar{q},\nabla\beta\rangle\right)+\\
&+2\,\mrm{Re}\left\langle(2+\norm{\beta}^2)\bar{q},\nabla^{1,0}_0\dot{\eta}\right\rangle-2\,\mrm{Re}\left\langle\nabla\bar{\beta},\nabla^{1,0} \dot{\eta}\right\rangle.
\end{split}
\end{equation*}

To prove that $(\dot{\eta},\varphi)\mapsto D\m{F}(\dot{\eta},\varphi)$ is an isomorphism we have to show that for any fixed $(\sigma,h)\in\m{A}^0(K_\Sigma)\times\m{C}^{\infty}(\Sigma)$ there is a unique pair $(\dot{\eta},\varphi)$ such that
\begin{equation}\label{eq:sistema_apertura}
\begin{dcases}
D\m{F}^1(\dot{\eta},\varphi)=\sigma\\
D\m{F}^2(\dot{\eta},\varphi)=h.
\end{dcases}
\end{equation}
Our strategy to prove this is to regard \eqref{eq:sistema_apertura} as a deformation of the system
\begin{equation}\label{eq:sistema_deformato}
\begin{dcases}
\frac{2\,{\nabla^{1,0}}^*\nabla^{1,0} \dot{\eta}}{1+\norm{q}^2} + \varphi\,\beta\frac{\norm{q}^2-1}{1+\norm{q}^2} =\sigma\\
\Delta(\varphi)+2\,\varphi(1-\norm{q}^2)=h.
\end{dcases}
\end{equation}
Since the two operators $\dot\eta\mapsto{\nabla^{1,0}}^*\nabla^{1,0}\dot{\eta}$ and $\varphi\mapsto\Delta(\varphi)+2\,\varphi$ are elliptic, self-adjoint and their kernel is trivial, it is straighforward to check that \eqref{eq:sistema_deformato} has a unique smooth solution $(\dot{\eta},\varphi)$ for each fixed $\sigma$, $h$. Now the equations \eqref{eq:sistema_apertura} differ from \eqref{eq:sistema_deformato} from terms which vanish as $\norm{\beta}_f$, $\norm{\nabla^{1,0}_f\beta}_f$, $\norm{\tau}_f$ and $\norm{\nabla^{1,0}_f\eta}_f$ go to zero; we have shown that all these terms can be bounded in terms of $\norm{\beta}_0$, $\norm{\nabla^{1,0} \beta}_0$, $\norm{\tau}_0$, effectively in terms of certain Sobolev and elliptic constants, so for $\norm{\beta}_{\m{C}^1}$ and $\norm{\tau}_0$ small enough we can make sure that \eqref{eq:sistema_apertura} also have a unique smooth solution. 

\begin{lemma}[Lemma $7.10$ in \cite{Fine_phd}]\label{lemma:inversione}
Let $D:B_1\to B_2$ be a bounded linear map between Banach spaces, with bounded inverse $D^{-1}$. Then any other linear bounded operator $L$ such that $||D-L||\leq (2\,||D^{-1}||)^{-1}$ is also invertible, and $||L^{-1}||\leq 2\,||D^{-1}||$.
\end{lemma}
In order to apply this result we regard $\m{F}$ as an operator
\begin{alignat*}{5}
\m{F}:\m{C}^{0,\alpha}(\Sigma,&K^2_\Sigma)\times\m{C}^{1,\alpha}(\Sigma,&&K_\Sigma)\times\m{C}^{2,\alpha}(\Sigma,&&K_\Sigma)\times\m{C}^{2,\alpha}(\Sigma,&&\bb{R})&&\to\m{C}^{0,\alpha}(\Sigma,K_\Sigma)\times\m{C}^{0,\alpha}(\Sigma,\bb{R})\\
&\tau &&\beta &&\eta &&\!\!f &&\mapsto(\m{F}^1(\tau,\beta,\eta,f),\m{F}^2(\tau,\beta,\eta,f))
\end{alignat*}
so that we are interested in the invertibility of the linear operator
\begin{alignat*}{3}
L:\m{C}^{2,\alpha}(\Sigma,&K_\Sigma)\times\m{C}^{2,\alpha}(\Sigma,&&\bb{R})&&\to\m{C}^{0,\alpha}(\Sigma,K_\Sigma)\times\m{C}^{0,\alpha}(\Sigma,\bb{R})\\
&\dot{\eta} &&\!\!\varphi &&\mapsto (D\m{F}^1_{(\tau,\beta,\eta,f)}(\dot{\eta},\varphi),D\m{F}^2_{(\tau,\beta,\eta,f)}(\dot{\eta},\varphi)).
\end{alignat*}
We compare $L$ to the auxiliary linear operator
\begin{alignat*}{3}
D:\m{C}^{2,\alpha}(\Sigma,&K_\Sigma)\times\m{C}^{2,\alpha}(\Sigma,&&\bb{R})&&\to\m{C}^{0,\alpha}(\Sigma,K_\Sigma)\times\m{C}^{0,\alpha}(\Sigma,\bb{R})\\
&\dot{\eta} &&\!\!\varphi &&\mapsto\left(\frac{2\,{\nabla^{1,0}}^*\nabla^{1,0}\dot{\eta}}{1+\norm{q}^2}+\varphi\,\beta\frac{\norm{q}^2-1}{1+\norm{q}^2},\Delta(\varphi)+2\,\varphi(1-\norm{q}^2)\right).
\end{alignat*}

$D$ is invertible, and the norm of $D^{-1}$ is controlled by the Schauder constants of $\Delta$ and ${\nabla^{1,0}}^*\nabla^{1,0}$. The difference between $D$ and $L$ is given by the operator
\begin{equation*}
\begin{split}
(D-L)(\dot{\eta},\varphi)=&\Bigg(\frac{-\beta}{1+\norm{q}^2}2\,\mrm{Re}\left\langle\nabla^{1,0}\dot{\eta},\bar{q}\right\rangle,\\
&\quad\quad 2\varphi\left(2\,\mrm{Re}\langle\bar{q},\nabla^{1,0}\beta\rangle-\norm{q}^2\norm{\beta}^2\right)+2\,\mrm{Re}\left\langle(2+\norm{\beta}^2)\bar{q},\nabla^{1,0}\dot{\eta}\right\rangle-2\mrm{Re}\left\langle\nabla\bar{\beta},\nabla^{1,0}\dot{\eta}\right\rangle\Bigg)
\end{split}
\end{equation*}
and we can estimate
\begin{equation*}
\begin{split}
\norm{D-L}\leq&\norm{\beta}_{1,\alpha}\left(1+\norm{q}_{0,\alpha}\right)\left(1+\norm{\beta}_{1,\alpha}\norm{q}_{0,\alpha}\right)+2\norm{q}_{0,\alpha}\left(1+\norm{\beta}_{1,\alpha}\right).
\end{split}
\end{equation*}
It is important to recall that in the present context all these norms are computed using the conformal metric $\omega_f$. However, our H\"older estimates on the conformal potential $f$ along the continuity method tell us that these norms are uniformly equivalent to those computed using the background hyperbolic metric $\omega$. Using also our H\"older estimates on $q$, it follows that we can control the norm of $D - L$, for $\alpha=1/2$, by the norms $\norm{\beta}_{\m{C}^{1,\frac{1}{2}}(\omega)}$, $\norm{\tau}_{\m{C}^{0,\frac{1}{2}}(\omega)}$. It these are small enough, then by Lemma \ref{lemma:inversione} the operator $L$ is invertible. Finally, bootstrapping shows that a solution of \eqref{eq:sistema_openness_f} in $\m{C}^{2,\alpha}$ is actually smooth.  

\addcontentsline{toc}{section}{References}
\bibliographystyle{abbrv}
{
\bibliography{bibliografia_HCSCK}
}
\vskip1cm
\noindent\rm{SISSA, via Bonomea 265, 34136 Trieste, Italy}\\
\noindent\tt{cscarpa@sissa.it}\\
\noindent\tt{jstoppa@sissa.it}
\end{document}